\documentclass[12pt]{amsart}

\usepackage{amssymb}
\newtheorem{theorem}{Theorem}[section]
\newtheorem{lemma}[theorem]{Lemma}
\newtheorem{proposition}[theorem]{Proposition}
\newtheorem{definition}[theorem]{Definition}

\numberwithin{equation}{section}

\begin{document}

\baselineskip=15.5pt

\title[Determinant line bundle of parabolic bundles]{Determinant
line  bundle on moduli space of parabolic bundles}

\author[I. Biswas]{Indranil Biswas}

\address{School of Mathematics, Tata Institute of Fundamental
Research, Homi Bhabha Road, Bombay 400005, India}

\email{indranil@math.tifr.res.in}



\date{}

\begin{abstract}
In \cite{BR1}, \cite{BR2}, a parabolic determinant
line bundle on a moduli space of stable parabolic bundles
was constructed, along with a Hermitian structure on it.
The construction of the Hermitian structure was indirect: The
parabolic determinant line bundle was identified with the pullback
of the determinant line bundle on a moduli space of
usual vector bundles
over a covering curve. The Hermitian structure on the
parabolic determinant bundle was taken to be the pullback of the
Quillen metric on the determinant line bundle on the moduli space of
usual vector bundles. Here a direct construction of the
Hermitian structure is given.
For that we need to establish a version of
the correspondence between the stable parabolic bundles and the
Hermitian--Einstein connections in the context of conical metrics.
Also, a recently obtained parabolic analog of Faltings' criterion
of semistability plays a crucial role.
\end{abstract}

\maketitle

\section{Introduction}

Let $X$ be a compact connected Riemann surface of genus $g$
and $S\, \subset\, X$ a finite nonempty subset; if $g\,=\,0$,
then $\# S\, \geq\, 3$. Fix a positive integer $r$. For each
point $s\,\in\, S$, fix sequences
$$
r\,=\,r_{s,1}\,>\, r_{s,2}\,>\, \cdots\,
> \, r_{s,\ell_s} \,> \, r_{s,\ell_s+1}\,=\, 0\, .
$$
and $0\,\le\, \alpha_1(s) \,<
\,\alpha_2(s) \,<\, \cdots \,<\, \alpha_{\ell_s}(s)\,<\, 1$.
Let ${\mathcal M}_P$ be the moduli space of sable parabolic
vector bundles of rank $r$ and degree $d$
over $X$ associated to this data.

In \cite{BR1}, \cite{BR2}, generalizing the determinant
line bundle over the moduli space of usual vector bundles, a
parabolic determinant line bundle was constructed which is a
holomorphic Hermitian line bundle over ${\mathcal M}_P$.
Fix a ramified Galois covering
$$
f\, :\, Y\, \longrightarrow\, X
$$
such that ${\mathcal M}_P$ is identified with a moduli space
${\mathcal M}_Y(\Gamma)$ of
stable $\Gamma$--linearized vector bundles, where $\Gamma\,=\,
\text{Gal}(f)$, over $Y$.

It was shown that the parabolic determinant line bundle is
holomorphically identified with the usual determinant line bundle
over ${\mathcal M}_Y(\Gamma)$. The Hermitian structure on
the parabolic determinant line bundle was obtained by
pulling back the Quillen metric on the usual determinant
line bundle over ${\mathcal M}_Y(\Gamma)$.
Our aim here is to give a direct construction of the
Hermitian structure on the parabolic determinant line bundle.

Fix a conical metric $h$ on $X$. We show that any stable
parabolic vector bundle $E_*\, \in\, {\mathcal M}_P$ admits
a unique Hermitian--Einstein connection with respect to
$h$ (see Theorem \ref{thm1}). This Hermitian--Einstein structure
plays a crucial role in the construction of the
Hermitian structure on the parabolic determinant line bundle.

Faltings' semistability criterion says that a vector bundle
$E$ on a smooth complex projective curve $C$ is semistable
if and only if there is a vector bundle
$F\,\longrightarrow\, C$ such that
$H^i(C,\, E\otimes F)\, =\, 0$ for $i\,=\, 0,1$ \cite{Fa}.
This criterion has the following generalization for parabolic
vector bundles.

Fix $X$ and parabolic divisor $S$ as above. Fix a positive
integer $N$; all
the parabolic weights will be assumed to be integral multiples
of $1/N$. There is a parabolic vector bundle
$V^0_*$ with the following property: a parabolic vector bundle
$E_*$ is parabolic semistable if and only if there is
another parabolic vector bundle $F_*$ such that
$$
H^i(X,\, (E_*\otimes F_*\otimes V^0_*)_0)\, =\, 0
$$
for $i\,=\, 0,1$, where $(E_*\otimes F_*\otimes V^0_*)_0$ is the
vector bundle underlying the parabolic tensor product
$E_*\otimes F_*\otimes V^0_*$. (See \cite{Bi2}.)

The above parabolic vector bundle $V^0_*$ is the other key
input in the construction of the
Hermitian structure on the parabolic determinant line bundle.

\section{Preliminaries}\label{sec1}

Let $X$ be a compact connected Riemann surface.
Fix a finite subset
\begin{equation}\label{e1}
\emptyset \,\not=\, S\, \subset\, X\, .
\end{equation}
A \textit{parabolic vector bundle} over $X$ with $S$
as the parabolic divisor is a
vector bundle $E$ on $X$ equipped with a parabolic
structure over $S$, meaning for each point $s\,\in\, S$,
we have
\begin{itemize}
\item{} a filtration
\begin{equation}\label{e2}
E_s\,=\,E_{s,1}\,\supset\, E_{s,2}\,\supset\, \cdots\,
\supset \,E_{s,\ell_s} \,\supset\, E_{s,\ell_s+1}\,=\, 0\, .
\end{equation}
of subspaces of the fiber $E_s$, and

\item rational numbers
\begin{equation}\label{e3}
0\,\le\, \alpha_1(s) \,<
\,\alpha_2(s) \,<\, \cdots \,<\, \alpha_{\ell_s}(s)\,<\, 1\, .
\end{equation}
\end{itemize}

The sequence in \eqref{e2} is called the \textit{quasiparabolic 
filtration}, and the sequence in \eqref{e3} is called the
\textit{parabolic weights}.
The points of $S$ are called \textit{parabolic points}.
For notational convenience, a vector bundle $E$ with
a parabolic structure will be denoted by $E_*$.

The \textit{parabolic degree} of $E_*$ is defined to be
$$
\text{par-deg}(E_*)\,:=\, \text{degree}(E)+\sum_{s\in S}
\sum_{i=1}^{\ell_s}\alpha_{\ell_s}(s)\cdot \dim E_{s,i}/E_{s,i+1}\, .
$$
Any subbundle $F$ of $E$ has an induced parabolic structure.
We recall that a parabolic vector bundle $E_*$ is called
\textit{semistable} if for all proper subbundles $F\, \subset\, E$
of positive rank, the inequality
$$
\frac{\text{par-deg}(F_*)}{\text{rank}(F_*)}\, \leq\,
\frac{\text{par-deg}(E_*)}{\text{rank}(E_*)}
$$
holds, where $F_*$ is the parabolic vector bundle defined by the
induced parabolic structure on $F$. If
the strict inequality
$$
\frac{\text{par-deg}(F_*)}{\text{rank}(F_*)}\, <\,
\frac{\text{par-deg}(E_*)}{\text{rank}(E_*)}
$$
holds then $E_*$ is called \textit{stable}.

A semistable parabolic vector bundle is called \textit{polystable}
if it is a direct sum of stable parabolic vector bundles.

Fix a positive integer $r$, and also fix an integer $d$. Let
${\mathcal M}_P$ denote the moduli space of stable parabolic
vector bundles of rank $r$ and degree $d$ with
a fixed type of quasiparabolic filtration and fixed parabolic
weights. (See \cite{MS} for the construction of ${\mathcal M}_P$
as well as some of its properties.) 

We will assume that $S$ and the parabolic data are such that
$\dim {\mathcal M}_P\, >\, 0$. Note that this
assumption rules out the
case where $X\,=\, {\mathbb C}{\mathbb P}^1$ and $\# S\,=\, 1$.
Indeed, if $X\,=\, {\mathbb C}{\mathbb P}^1$ with $\# S\,=\, \{s\}$,
then using the Grothendieck's theorem that any vector bundle over 
${\mathbb C}{\mathbb P}^1$ splits into a direct sum of line bundles
it follows that any semistable parabolic vector bundle $E_*$ 
must be of the following form:
$$
E\, =\,L^{\oplus r}\, ,
$$
where $L\, \longrightarrow\, {\mathbb C}{\mathbb P}^1$ is a line
bundle, and the quasiparabolic filtration is the trivial
filtration
$$
0\,\subset\, E_s\, .
$$

A theorem due to Faltings says that a vector bundle $E$ over
$X$ is semistable if and only if there is another vector bundle
$E'$ such that $E\otimes E'$ is cohomologically trivial, meaning
$$
H^0(X,\, E\otimes E')\,=\, 0\, =\, H^1(X,\, E\otimes E')
$$
\cite{Fa}. This criterion can be generalized to the context of
parabolic vector bundles in the following way. Fix a positive
integer $N$. Consider parabolic vector bundles with parabolic
weights integral multiples of $1/N$ (with arbitrary quasiparabolic
structure and rank).

\begin{lemma}[\cite{Bi2}]\label{lem1}
Then there is a parabolic vector bundle
$V^0_*$ with the following property:
A parabolic vector bundle $E_*$ is semistable if and only
if there is a parabolic vector bundle $E'_*$ such that the vector
bundle $(E_*\otimes E'_*\otimes V^0_*)_0$ underlying
the parabolic tensor product $E_*\otimes E'_*\otimes V^0_*$ is
cohomologically trivial.
\end{lemma}

It should be emphasized that such a parabolic vector bundle $V^0_*$ 
can be explicitly constructed.
We recall a construction of a parabolic vector bundle $V^0_*$
that satisfies the above condition.

Fix a Galois covering
\begin{equation}\label{f}
f\,:\, Y\, \longrightarrow\, X
\end{equation}
such that
\begin{itemize}
\item $f$ is ramified exactly over $S$, and

\item the ramification index of each point in $f^{-1}(S)$ is
$N-1$, where $N$ is the fixed integer such that all the parabolic
weights are integral multiples of $1/N$.
\end{itemize}
See \cite[p. 26, Proposition 1.2.12]{Na} for the
existence of $f$ satisfying these conditions. Let
\begin{equation}\label{gl}
\Gamma\, :=\, \text{Gal}(f)
\end{equation}
be the Galois group for the covering $f$.
There is a natural bijective correspondence between the
following two classes:
\begin{enumerate}
\item all $\Gamma$--linearized vector bundles on $Y$, and

\item the parabolic vector bundles over $X$ for which the
parabolic divisor in contained in $S$ and all the
parabolic weights are integral multiples of $1/N$.
\end{enumerate}
(See \cite{Bi1}.)

Consider the trivial vector bundle over $Y$
\begin{equation}\label{hW}
\widetilde{V}\, :=\, {\mathcal O}_Y\otimes_{\mathbb C}
{\mathbb C}(\Gamma)\, ,
\end{equation}
where ${\mathbb C}(\Gamma)$ is the group algebra of $\Gamma$
defined in \eqref{gl}. The action of $\Gamma$ on $Y$ produces
an action of $\Gamma$ on ${\mathcal O}_Y$.
The natural action of $\Gamma$
on ${\mathbb C}(\Gamma)$ and the action of $\Gamma$
on ${\mathcal O}_Y$ together define a 
$\Gamma$--linearization on the vector bundle $\widetilde{V}$
in (\ref{hW}). Let
\begin{equation}\label{V}
V^0_*\, \longrightarrow\, X
\end{equation}
be the parabolic vector bundle over $X$ corresponding
to the $\Gamma$--linearized vector bundle $\widetilde{V}$.
This parabolic vector bundle $V^0_*$ satisfies the condition
in Lemma \ref{lem1}. (See \cite{Bi2}, \cite{BH}.)

\section{Conical metric on $X$}\label{sec.cm}

Fix a positive integer $N$.

Let $S$ be the subset in \eqref{e1}. Set
$$
X'\,:=\, X\setminus S\, .
$$

Let ${\mathbb D}\, :=\,\{z\,\in\, {\mathbb C}\, \mid\,
\vert z \vert\, <\, 1\}$ be the unit disk. For a point
$x\, \in\, X$, a \textit{holomorphic
coordinate} around $x$ is a holomorphic embedding
$$
\varphi\, :\, {\mathbb D}\, \longrightarrow\, X
$$
such that $\varphi(x)\,=\, 0$.

A \textit{conical metric} on $X$ of order $N$ is a Hermitian
metric $h$ on the holomorphic tangent bundle $TX'$ satisfying
the following condition: For each point $x\,\in\, D$,
there is a holomorphic coordinate around $x$
$$
\varphi\, :\, {\mathbb D}\, \longrightarrow\, X
$$
such that
$$
h\vert_{\mathbb D}\, =\, f(z)\frac{dz\otimes
d\overline{z}}{|z|^{2(N-1)/N}}\, ,
$$
where $f\, :\, {\mathbb D} \, \longrightarrow\,
{\mathbb R}^+$ is a smooth function; $\varphi({\mathbb D})$
is identified with ${\mathbb D}$ using $\varphi$.

Let $E_*$ be a parabolic vector bundle over $X$. Let
$E$ be the vector bundle underlying the parabolic bundle $E_*$.

A \textit{Hermitian structure} on $E_*$ is a Hermitian structure
$H$ on $E\vert_{X'}$ satisfying the following condition:
Take any point $s\, \in\, S$, and take any holomorphic section
$\sigma$ of $E$ defined around $s$. If $\sigma(s)$ is a nonzero
element of $E_{s,i}\, \subset\, E_s$ (see \eqref{e2}), then
$$
\Vert s\Vert_H \,=\, f(z) \vert z\vert^{\alpha_i(s)}
$$
(see \eqref{e3} for $\alpha_i(s)$), where $z$ is a
holomorphic coordinate around $s$, and $f$ is a smooth
function with values in positive real numbers.

Fix a conical metric of order $N$ on $X$. Let $\omega$ be
the corresponding K\"ahler form on $X'$.

Let $E_*$ be a parabolic vector bundle over $X$ with the
property that all the parabolic weights of $E_*$ are
integral multiples of $1/N$.

\begin{definition}\label{HE}
{\rm A Hermitian structure $H$ on $E_*$ is called}
Hermitian--Einstein {\rm if the curvature of the Chern
connection on $E\vert_{X'}$ for $H$ is of the form}
$$
\lambda\cdot {\rm Id}_E\otimes \omega\, ,
$$
{\rm where $\lambda$ is some constant positive real number.}
\end{definition}

\begin{theorem}\label{thm1}
Let $E_*$ be a stable parabolic vector bundle over $X$
such that all the parabolic weights of $E_*$ are
integral multiples of $1/N$. Then $E_*$ admits a
Hermitian--Einstein structure. If $H$ and $H_1$ are
two Hermitian--Einstein structures on $E_*$, then
$$
H_1\, =\, c\cdot H\, ,
$$
where $c$ is a constant real positive number.
\end{theorem}

\begin{proof}
Let $f\, :\, Y\, \longrightarrow\, X$ be the covering
in \eqref{f}. The pullback of the (fixed) conical metric
on $X$ defines a Hermitian structure on
$f^{-1}(X\setminus S)$. From the definition of the
conical metric it follows that this Hermitian metric
on $f^{-1}(X\setminus S)$ extends to a Hermitian metric
on $Y$. Let $\widetilde{\omega}$
be the K\"ahler form on $Y$ associated to this
Hermitian structure on $Y$.

Let
$$
W\, \longrightarrow\, Y
$$
be the $\Gamma$--linearized vector bundle corresponding
to the parabolic vector bundle $E_*$. Since $E_*$ is
parabolic stable, it follows that the corresponding
$\Gamma$--linearized vector bundle
$W$ is polystable \cite[p. 349, Proposition 4.1]{BBN}.
Consequently, $W$ admits a Hermitian--Einstein structure
\cite{Do}.
The Hermitian--Einstein structure is not unique, but the
Hermitian--Einstein connection is unique. From the
uniqueness of the Hermitian--Einstein connection it follows
immediately that the action of $\Gamma$ on $W$ preserves
the Hermitian--Einstein connection. Consequently, for any
Hermitian--Einstein structure $H_W$ on $W$, the Hermitian
form
$$
H'_W\, :=\, \sum_{\gamma\in\Gamma} \gamma^* H_W
$$
is Hermitian--Einstein. Clearly, the action of $\Gamma$ on $W$
preserves the Hermitian form $H'_W$. Therefore, $H'_W$ descends
to a Hermitian structure on $E\vert_{X\setminus S}$. It is now
straight--forward to check that this Hermitian structure on 
$E\vert_{X\setminus S}$ is a Hermitian structure on the parabolic
vector bundle $E_*$.

{}From the definition of a Hermitian--Einstein structure on $E_*$
it follows that any two Hermitian--Einstein forms on $E_*$ differ
by an automorphism of the parabolic vector bundle $E_*$.
Since $E_*$ is parabolic stable, all parabolic automorphisms
are constant scalar multiplications. Hence any two
Hermitian--Einstein forms on $E_*$ differ by multiplication
with a constant real number. This completes the proof of the
theorem.
\end{proof}

\section{The determinant line bundle}\label{sec2}

Let ${\mathcal E}\, \longrightarrow\, X\times T$ be a holomorphic
vector bundle, where $T$ is a complex manifold. We will consider
$\mathcal E$ as a holomorphic family of vector bundles over $X$ 
parametrized by $T$. Let
\begin{equation}\label{p}
p\, :\, X\times T\, \longrightarrow\, T
\end{equation}
be the projection. The direct images $R^0p_*\mathcal E$
and $R^1 p_*\mathcal E$ are coherent analytic sheaves on $T$.
Therefore,
$$
\det (R^i p_*\mathcal E)\, :=\, \bigwedge\nolimits^{\rm top}
R^i p_*\mathcal E\, ,
$$
$i\,=\, 0\, ,1$, are holomorphic line bundles over $T$ (see
\cite[Ch. V, \S~6]{Ko} for the construction of the determinant
line bundle of a coherent analytic sheaf). If $T$ is algebraic,
and $\mathcal E$ is an algebraic vector bundle, then the
line bundle $\det (R^i p_*\mathcal E)$ is also algebraic.

The \textit{determinant} of the family $\mathcal E$ is
defined to be the line bundle
\begin{equation}\label{d}
d({\mathcal E})\, :=\, \det (R^0 p_*\mathcal E)^*\otimes
\det (R^1 p_*\mathcal E)\, \longrightarrow\, T\, .
\end{equation}

Let
\begin{equation}\label{e5}
{\mathcal E}_*\, \longrightarrow\, X\times T
\end{equation}
be a family of parabolic vector bundles of fixed quasiparabolic
type and fixed parabolic weights. Let
\begin{equation}\label{px}
p_X\, :\, X\times T\, \longrightarrow\, X
\end{equation}
be the natural projection. Consider the parabolic
vector bundle $V^0_*$ constructed in \eqref{V}.
So $p^*_XV^0_*$ is a constant family of parabolic vector bundles
parametrized by $T$. Let
\begin{equation}\label{e4}
{\mathcal E}_*\otimes p^*_XV^0_*\, \longrightarrow\, X\times T 
\end{equation}
be the family of parabolic vector bundles obtained by taking
the parabolic tensor product. Let
$$
({\mathcal E}_*\otimes p^*V^0_*)_0\, \longrightarrow\, X\times T 
$$
be the family of vector bundles underlying the family of
parabolic vector bundles in \eqref{e4}.

\begin{definition}\label{def1}
{\rm The} parabolic determinant bundle {\rm for the family
${\mathcal E}_*$ is defined to be the line 
bundle}
$$
d(({\mathcal E}_*\otimes p^*V^0_*)_0)\, \longrightarrow\, T
$$
{\rm (see \eqref{d}). The parabolic determinant bundle for
${\mathcal E}_*$ will be denoted by}
$pd({\mathcal E}_*)$.
\end{definition}

Let
${\mathcal M}_P$ be a moduli space of stable parabolic
vector bundles over $X$ of rank $r$ and degree $d$ with
a fixed type of quasiparabolic filtration and fixed parabolic
weights. In general, there is no universal parabolic vector
bundle over $X\times {\mathcal M}_P$. However, every point
$z\, \in\, {\mathcal M}_P$ has an open neighborhood $U_z$ in
\'etale topology such that there is a
universal parabolic vector bundle over $X\times U_z$. Hence
there is an open neighborhood $U'_z$ of $z$ in
analytic topology such that there is a
universal parabolic vector bundle over $X\times U'_z$.

Using the locally defined universal parabolic vector bundles
on ${\mathcal M}_P$,
we can construct locally defined parabolic determinant bundles.

Fix a point $x_0\,\in\, X$. Let ${\mathcal E}_*$ be 
a locally defined (in either \'etale or analytic 
topology) universal parabolic vector bundles over $X\times 
U$. The vector bundle over $U$ obtained by 
restricting the underlying vector bundler ${\mathcal E}$
to $\{x_0\}\times U$ will be denoted by
${\mathcal E}_{x_0}$.

Assume that $T$ is connected. Hence the function
${\mathcal M}_P\,\longrightarrow\, \mathbb Z$
that sends any $E_*\, \in\, {\mathcal M}_P$ to
\begin{equation}\label{chi}
\chi\,:=\, \dim H^0(X,\, E)- \dim H^1(X,\, E)
\end{equation}
is a constant one.

\begin{lemma}\label{lem2}
The locally defined line bundles
$$
pd({\mathcal E}_*)^{\otimes r}\otimes
(\bigwedge\nolimits^r {\mathcal E}_{x_0})^{\otimes \chi}
$$
(see Definition \ref{def1}) patch together naturally to
define an algebraic line bundle over ${\mathcal M}_P$.
\end{lemma}

\begin{proof}
Let $E_*$ be a stable parabolic vector bundle. Take any
automorphism $\tau$ of the underlying vector bundle $E$ that
preserves the quasiparabolic filtrations. Since $E_*$ is stable,
we know that $\tau\,=\, \lambda\cdot \text{Id}_E$ for some
$\lambda\,\in\, {\mathbb C}^*$. Using this it follows that
if ${\mathcal E}_*$ and ${\mathcal E}'_*$ are two universal parabolic 
vector bundles over $X\times U$, then there is a natural
line bundle $L\, \longrightarrow\, U$ and a
canonical isomorphism
\begin{equation}\label{e6}
{\mathcal E}'_*\,=\, {\mathcal E}_*\otimes f^*L\, ,
\end{equation}
where $\phi\, :\, X\times U\, \longrightarrow\, U$ is the 
projection. In fact, we may take
$$
L\, :=\, \phi_*(Hom({\mathcal E}_*\, ,{\mathcal E}'_*))
$$
(here $Hom$ is the sheaf of parabolic homomorphisms).
In this case, the isomorphism in \eqref{e6} is given by the
natural pairing
$$
Hom({\mathcal E}_*\, ,{\mathcal E}'_*)\otimes {\mathcal E}_*\, 
\longrightarrow\,{\mathcal E}'_*\, .
$$

{}From \eqref{e6} and the projection formula
it follows that 
\begin{equation}\label{e7}
pd({\mathcal E}_*)\, =\, pd({\mathcal E}'_*)\otimes
L^{\otimes \chi}\, ,
\end{equation}
where $\chi$ is defined in \eqref{chi}.
On the other hand, from \eqref{e6},
\begin{equation}\label{e8}
\bigwedge\nolimits^r {\mathcal E}'_{x_0}
\,=\, (\bigwedge\nolimits^r {\mathcal E}_{x_0})
\otimes L^{\otimes r}\, .
\end{equation}
{}From \eqref{e7} and \eqref{e8},
$$
pd({\mathcal E}_*)^{\otimes r}\otimes (\bigwedge\nolimits^r 
\bigwedge\nolimits^r 
{\mathcal E}_{x_0})^{\otimes \chi}\,=\,
pd({\mathcal E}'_*)^{\otimes r}\otimes (\bigwedge\nolimits^r 
\bigwedge\nolimits^r {\mathcal E}'_{x_0})^{\otimes \chi}\, .
$$
This completes the proof of the lemma.
\end{proof}

\begin{definition}\label{D}
{\rm Let}
$$
{\mathcal D}\, \longrightarrow\,{\mathcal M}_P
$$
{\rm be the holomorphic line bundle obtained from Lemma \ref{lem2}.}
\end{definition}

There is a parabolic determinant line bundle on ${\mathcal M}_P$;
see \cite{BR1} and \cite{BR2} for its construction.

\begin{lemma}\label{lem-d}
The holomorphic
line bundle ${\mathcal D}$ in Definition \ref{D} coincides
with the parabolic determinant line bundle.
\end{lemma}

\begin{proof}
Take any parabolic vector bundle $E_*\, \in\, {\mathcal M}_P$.
Let
$$
W\, \longrightarrow\, Y
$$
be the $\Gamma$--linearized vector bundle
corresponding to $E_*$, where $Y$ is the Galois
covering in \eqref{f}. For $i\,=\, 0\, ,1$, we have
\begin{equation}\label{i}
H^i(X,\, ({\mathcal E}_*\otimes p^*V^0_*)_0)\,=\,
H^i(Y,\, W)
\end{equation}
\cite[p. 252, (15)]{BH} (in \cite[p. 252, (15)]{BH} it is
proved for $i\,=\,0$, but the proof is identical for $i\,=\,1$;
see \cite[p. 327, (8)]{Bi2}). Using \eqref{i} it is
straight--forward to check that ${\mathcal D}$ coincides
with the parabolic determinant line bundle.
\end{proof}

\section{Hermitian structures}

\subsection{Hermitian structure on $V^0$}

Let $V^0$ denote the vector bundle underlying the parabolic
bundle $V^0_*$ in \eqref{V}.
We will first construct a Hermitian structure on the
restriction of $V^0$ to
\begin{equation}\label{xp}
X'\,:=\, X\setminus S\, \subset\, X\, .
\end{equation}

Consider the group $\Gamma$ in \eqref{gl}.
Let $h_\Gamma$ be the inner product on the group
algebra ${\mathbb C}(\Gamma)$ defined by
$$
h_\Gamma(\sum_{z\in\Gamma} c_z\cdot z\, ,\sum_{z\in\Gamma}
d_z\cdot z)\,=\, \sum_{z\in\Gamma} c_z\overline{d_z}\, .
$$
This inner product
$h_\Gamma$ defines a Hermitian structure on the trivial
vector bundle $\widetilde{V}\,=\, {\mathcal O}_Y\bigotimes_{\mathbb C}
{\mathbb C}(\Gamma)$ in \eqref{hW}. The action of the Galois
group $\Gamma\,=\, \text{Gal}(f)$ on $\widetilde{V}$ preserves
this Hermitian structure.
The pullback $f^*(V^0\vert_{X'})$ is identified
with the restriction of $\widetilde{V}$ to $f^{-1}(X')\,\subset\,
Y$, where $X'$ is the open subset in \eqref{xp}. Consequently,
there is a unique Hermitian structure on $V^0\vert_{X'}$
such that the identification of $f^*(V^0\vert_{X'})$ with
$\widetilde{V}\vert_{f^{-1}(X')}$ is an isometry.

\begin{definition}\label{h-v}
{\rm Let $h_0$ denote the Hermitian structure on $V^0\vert_{X'}$ 
constructed above.}
\end{definition}

It is straight--forward to check that $h_0$ is a Hermitian structure 
on the parabolic vector bundle $V^0_*$.

The parabolic vector bundle $V^0_*$ is polystable. Hence it has
a Hermitian--Einstein structure (see Theorem \ref{thm1}). The following
proposition is obtained by comparing the constructions of the
Hermitian structure $h_0$ (see Definition \ref{h-v}) and
Hermitian--Einstein structure in Theorem \ref{thm1}.

\begin{proposition}\label{prop.-i}
The Hermitian structure $h_0$ on the polystable parabolic vector
bundle $V^0_*$ is a Hermitian--Einstein form.
\end{proposition}

\subsection{Hermitian structure on $\mathcal D$}

Fix a positive integer $N$ such that all the parabolic weights
in the parabolic data associated to ${\mathcal M}_P$ are
integral multiples of $1/N$. 
Fix a conical metric $h$ of order $N$ on $X$.

Let $E_*\, \in\, {\mathcal M}_P$ be a stable parabolic
vector bundle. Let $H$ be a Hermitian--Einstein form
on $E_*$. From Theorem \ref{thm1} we know that $H$ is
determined uniquely up to a constant scalar.

The Hermitian form $h_0$ on $V^0\vert_{X'}$ (see
\eqref{h-v}) and the Hermitian--Einstein form $H$
on $E_*$
together define a Hermitian--Einstein structure on
$(E\otimes V^0)\vert_{X'}$, where $X'$ is
defined in \eqref{xp}. This Hermitian structure
on $(E\otimes V^0)\vert_{X'}$ defines a
Hermitian structure on the parabolic tensor product
$E_*\otimes V^0_*$. Let $\widehat{H}$ denote the
Hermitian structure on the parabolic vector bundle
$E_*\otimes V^0_*$ given by the above
Hermitian structure on $(E\otimes V^0)\vert_{X'}$.

Let $(E_*\otimes V^0_*)_0$ be the vector bundle
underlying the parabolic vector bundle
$E_*\otimes V^0_*$.

The Hermitian structure $\widehat{H}$ and
the conical metric $h$ together define a
Hermitian structure on the vector space
$H^0(X,\, (E_*\otimes V^0_*)_0)$.

The Hermitian structure $\widehat{H}$
induces a Hermitian structure on
$$
C^\infty(X;\, (E_*\otimes V^0_*)_0\otimes
\Omega^{0,1}_X)\, .
$$
The restriction of it to the orthogonal complement
of the image of the Dolbeault operator for the
holomorphic vector bundle $(E_*\otimes V^0_*)_0$
\begin{equation}\label{dol}
\overline{\partial}_{(E_*\otimes V^0_*)_0}\, :\,
C^\infty(X;\, (E_*\otimes V^0_*)_0)\,
\longrightarrow\, C^\infty(X;\, (E_*\otimes V^0_*)_0
\otimes \Omega^{0,1}_X)
\end{equation}
defines a Hermitian structure on the vector space
$H^1(X,\, (E_*\otimes V^0_*)_0)$.

Let
\begin{equation}\label{l}
\Delta_{(E_*\otimes V^0_*)_0}\, :=\,
(\overline{\partial}_{(E_*\otimes V^0_*)_0})^*
\overline{\partial}_{(E_*\otimes V^0_*)_0}
\end{equation}
be the Laplacian of the operator in \eqref{dol}.

Consider the complex line
$$
pd(E_*)\, :=\, \bigwedge\nolimits^{\text{top}}H^0(X,\, (E_*
\otimes V^0_*)_0)^* \otimes \bigwedge\nolimits^{\text{top}}
H^1(X, \,(E_*\otimes V^0_*)_0)
$$
(compare it with Definition \ref{def1}).
The above inner products on $H^0(X,\, (E_*\otimes
V^0_*)_0)$ and $H^1(X,\, (E_*\otimes V^0_*)_0)$ together produce
an inner product on $pd(E_*)$. Using the Quillen's construction,
we modify this inner product on $pd(E_*)$ using the eigenvalues of
the Laplacian $\Delta_{(E_*\otimes V^0_*)_0}$ defined in \eqref{l}
(see \cite{Qu}). This procedure
produces a Hermitian structure on the holomorphic line bundle
$pd({\mathcal E}_*)\, \longrightarrow\, T$ (see Definition
\ref{def1}). Hence we get a Hermitian structure on the holomorphic
line bundle
$$
{\mathcal D}\, \longrightarrow\,{\mathcal M}_P
$$
constructed in Definition \ref{D}.

\begin{definition}\label{def2}
{\rm The above Hermitian structure on the holomorphic line
bundle ${\mathcal D}\, \longrightarrow\,{\mathcal M}_P$ will be
denoted by $H_{pQ}$.}
\end{definition}

Consider the covering $f$ in \eqref{f}. As noted in the proof
of Theorem \ref{thm1}, the pullback of the conical metric $h$
by $f$ produces a Hermitian metric on $Y$. As before, the
K\"ahler form on $Y$ associated to this Hermitian metric
will be denoted by $\widetilde{\omega}$.

Let
$$
W\, \longrightarrow\, Y
$$
be the $\Gamma$--linearized vector bundle
corresponding to $E_*$, where $Y$ is the
Galois covering in \eqref{f}. Let
$$
\overline{\partial}_W\, :\, C^\infty(Y;\, W)
\,\longrightarrow\, C^\infty(Y;\, W\otimes
\Omega^{0,1}_Y)
$$
be the Dolbeault operator for the holomorphic
vector bundle $W$. Since the vector
bundle $W$ is polystable, it has a Hermitian--Einstein
structure. Using this Hermitian--Einstein structure
and the K\"ahler form $\widetilde{\omega}$ on $Y$ we
define the Laplacian
$$
\Delta_{W}\, :=\,\overline{\partial}^*_W
\overline{\partial}_W\, .
$$
It can be shown that the eigenvalues of $\Delta_{W}$,
along with their multiplicities, coincide with those of
the operator $\Delta_{(E_*\otimes V^0_*)_0}$ constructed
in \eqref{l}. Indeed,
$$
C^\infty(Y;\, W) ~\,~\,~\,\text{~and~}\,~\,~\,~
C^\infty(Y;\, W\otimes\Omega^{0,1}_Y)
$$
are identified with $C^\infty(X;\,(E_*\otimes V^0_*)_0
\otimes\Omega^{0,1}_X)$ and
$C^\infty(X;\, (E_*\otimes V^0_*)_0)$ respectively,
and these identifications preserve the inner products.
Furthermore, these identifications take the differential
operator $\overline{\partial}_{(E_*\otimes V^0_*)_0}$ to 
$\overline{\partial}_W$. Consequently, these identification
takes $\Delta_{(E_*\otimes V^0_*)_0}$ to $\Delta_{W}$. Hence
the eigenvalues of $\Delta_{W}$, along with their
multiplicities, coincide with those of the operator
$\Delta_{(E_*\otimes V^0_*)_0}$.

Therefore, we have the following proposition:

\begin{proposition}\label{prop2}
The holomorphic isomorphism in Lemma \ref{lem-d} between
${\mathcal D}$ and the parabolic determinant line bundle
takes in Hermitian structure $H_{pQ}$ in Definition \ref{def2}
to the Hermitian structure on the parabolic determinant line bundle.
\end{proposition}

See \cite{BR1}, \cite{BR2} for the construction of the
Hermitian structure on the parabolic determinant line bundle.


\end{document}